\documentclass[english,a4paper]{amsart}
\usepackage[T1]{fontenc}
\usepackage[latin1]{inputenc}
\usepackage{graphicx}
\usepackage{amssymb}

\makeatletter

 \theoremstyle{plain}
\newtheorem{thm}{Theorem}[section]
  \theoremstyle{plain}
  \newtheorem{prop}[thm]{Proposition}
  \theoremstyle{plain}
  \newtheorem{lem}[thm]{Lemma}
 \theoremstyle{definition}
 \newtheorem*{defn*}{Definition}
  \theoremstyle{plain}
  \newtheorem{cor}[thm]{Corollary}

\newcommand{\e}{\mathrm e}
\DeclareMathOperator{\Int}{Int}

\renewcommand{\hat}{\widehat}
\renewcommand{\phi}{\varphi}

\renewcommand{\tilde}{\widetilde}

\def\diam {\mathop {\hbox{\rm diam}}}
\def\supp {\mathop {\hbox{\rm supp}}}

\def\R{\mathbb{R}}

\def\N{\mathbb{N}}

\usepackage{mathptmx}

\makeatother
\usepackage{babel}

\begin{document}

\title{Hölder-differentiability of Gibbs distribution functions }

\date{November 28, 2007}

\author{Marc Kesseböhmer and Bernd O. Stratmann}



\begin{abstract}
In this paper we give non-trivial applications of the thermodynamic
formalism to the theory of distribution functions of Gibbs measures
(devil's staircases) supported on limit sets of finitely generated
conformal iterated function systems in $\R$. For a large class of
these Gibbs states we determine the Hausdorff dimension of the set
of points at which the distribution function of these measures is
not $\alpha$-Hölder-differentiable. The obtained results give significant
extensions of recent work by Darst, Dekking, Falconer, Li, Morris,
and Xiao. In particular, our results clearly show that the results
of these authors have their natural home within thermodynamic formalism. 
\end{abstract}

\address{Fachbereich 3 - Mathematik und Informatik, Universität Bremen, D--28359
Bremen, Germany}

\email{mhk@math.uni-bremen.de}

\address{Mathematical Institute, University of St Andrews, St Andrews KY16
9SS, Scotland}

\email{bos@maths.st-and.ac.uk}

\maketitle

\section{Introduction}

In this paper we study the limit set $\mathcal{L}$ of an iterated
function system generated by a finite set of conformal contractions
$\{f_{a}:a\in A\}$ in $\R$ satisfying the strong separation condition.
It is well known that each suitably chosen potential function $\psi$
on $\mathcal{L}$ gives rise to a Gibbs measure $\nu_{\psi}$ supported
on $\mathcal{L}$. For instance, for the geometric potential $\phi:=\log f_{a}'$
and with $\delta$ referring to the Hausdorff dimension of $\mathcal{L}$,
we have that the Gibbs measure $\nu_{\delta\phi}$ is in the same
measure class as the $\delta$-dimensional Hausdorff measure on $\mathcal{L}$.
In this paper we concentrate on Gibbs measures $\nu_{\psi}$ associated
with Hölder-continuous potential functions $\psi$ for which $P\left(\psi\right)=0$
and $\alpha\phi<\psi<0$, for some fixed $\alpha\in\R_{+}$. Here,
$P$ refers to the usual pressure function associated with $\mathcal{L}$
(see Section $2$ for the definition). For potentials of this type,
we consider the set $\Lambda_{\psi}^{\alpha}$ of points at which
the $\alpha$-Hölder derivative of the distribution function $F_{\psi}$
of $\nu_{\psi}$ does not exist in the generalized sense (note, $F_{\psi}$
is an `ordinary devil's staircase'). That is, for $\alpha\in\R_{+}$
we consider the set \[
\Lambda_{\psi}^{\alpha}:=\{\xi\in\mathcal{L}:\left(D^{\alpha}F_{\psi}\right)(\xi)\;\mbox{does \, neither \, exist \, nor \, is \, equal \, to \, infinity}\},\]
 where $D^{\alpha}$ refers to the $\alpha$-Hölder derivative defined
for functions $F$ on $\mathcal{L}$ by (given that the limit exists)
\[
\left(D^{\alpha}F\right)(\xi):=\lim_{\eta\to\xi}\frac{\left|F\left(\xi\right)-F\left(\eta\right)\right|}{\left|\xi-\eta\right|^{\alpha}},\quad\hbox{for}\,\,\xi\in\R.\]
 We show that the Hausdorff dimension $\dim_{H}(\Lambda_{\psi}^{\alpha})$
of $\Lambda_{\psi}^{\alpha}$ can be determined by employing the thermodynamic
formalism. The main results of the paper are summarized in the following
theorem.

\vspace{2mm}

\noindent \textbf{Main Theorem.} \emph{ Let $\mathcal{L}$ and $\psi$
be given as above. Then the Hausdorff dimension of $\Lambda_{\psi}^{\alpha}$
is given by \begin{equation}
\dim_{H}(\Lambda_{\psi}^{\alpha})=s,\label{main}\end{equation}
 where $s$ is the unique solution of the equation \begin{equation}
\beta_{\alpha}(s)+s\cdot\min\{\phi(\underline{i})/\psi(\underline{i}):i\in\{0,1\}\}=0.\label{eq:linRel}\end{equation}
 Here, $\beta_{\alpha}$ is determined implicitly by the pressure
equation \begin{equation}
P((t-\alpha\beta_{\alpha}(t))\phi+\beta_{\alpha}(t)\psi)=0\,\,\hbox{for}\,\, t\in\R,\label{pressure}\end{equation}
 and the symbol $i=0$ ($1$ resp.) refers to the letter in the alphabet
used to code the utter left (right resp.) interval in the geometric
representation of the iterated function system, and $\underline{i}$
denotes the infinite word which exclusively contains the letter $i$.}

\vspace{2mm}

\noindent \textbf{Remarks.}

\noindent \textbf{I.} Let us remark that our Main Theorem generalizes
recent work in \cite{Darst:95,DekkingWenxia:03,Falconer:04,LiXiaoDekking:02,Li:07,Mo:02}.
In comparision to the approaches of these authors, with the slight
exception of \cite{Falconer:04} who at least employed multifractal
analysis in his study of the Ahlfors regular case, in this paper we
develop a completely different and much more general approach which
gives these results their natural home within the conceptionally wider
frame of the thermodynamical formalism. In fact, we combine certain
techniques from this formalism (see Section 2 for the details) with
certain other techniques which have their origins in metric Diophantine
analysis. By the latter we mean those techniques which were derived
through generalizations of results by Jarn\'{\i}k \cite{Jarnik:29}
and Besicovitch \cite{Besicovitch:34} on well-approximable irrational
numbers to cuspital excursions on hyperbolic manifolds (see e.g. \cite{S:95,HillVelani:98,S:99}),
and to Julia sets of parabolic rational maps (see e.g. \cite{SU:02}).

\vspace{2mm}

\noindent \textbf{II.} Let us also remark that the results in this
paper can be expressed in terms of so called `or\-dinary de\-vil's
staircases' as follows. For this recall that the distribution function
of a non-atomic positive finite Borel measure $\mu$ on a compact
interval in $\R$ is a non-increasing continuous function which is
constant on the complement of $\supp(\mu)$, the support of $\mu$.
Such a distribution function is called an {\em ordinary devil's
staircase} if the $1$-dimensional Lebesgue measure $\lambda(\supp(\mu))$
of $\supp(\mu)$ vanishes. Obviously, the distribution functions $F_{\psi}$
which we consider in this paper \textsl{are} ordinary devil's staircases.
Interesting sets for devil's staircases are the set $\Delta_{0}$
of points where the staircase has derivative equal to zero, the set
$\Delta_{\infty}$ where the derivative is equal to infinity, and
the set $\Delta_{\sim}$ where the derivative does not exist. Clearly,
for the type of staircases in this paper we trivially have that $\lambda$-almost
every point is in $\Delta_{0}$, and hence $\dim_{H}(\Delta_{0})=1$
(in fact, this also holds for the slippery devil's staircases below
(see e.g. \cite[section 31]{Bil:79})). Also, combining our Main Theorem
and Corollary \ref{cor1} in this paper, we (almost) immediately have
$\dim_{H}(\Delta_{\infty})=\dim_{H}(\mathcal{L})$. Therefore, for
the type of ordinary devil's staircases in this paper we have \begin{equation}
\dim_{H}(\Delta_{\sim})<\dim_{H}(\Delta_{\infty})<\dim_{H}(\Delta_{0})=1.\label{devil}\end{equation}
 Here the question arises of how the Hausdorff dimension $\dim_{H}(\nu_{\psi})$
of the measure $\nu_{\psi}$ fits into this picture. In fact, for
the Darst self-similar case (see Remark III. below) we found numerically
that there are cases in which $\dim_{H}(\nu_{\psi})>\dim_{H}(\Delta_{\sim})$
(cf. Fig.~1) as well as cases where $\dim_{H}(\nu_{\psi})<\dim_{H}(\Delta_{\sim})$
(cf. Fig.~2) . Therefore, there is no hope to include $\dim_{H}(\nu_{\psi})$
into the hirachy of dimensions in (\ref{devil}) in general.\\
 Note that these results are in slight contrast to our results for
a certain slippery devil's staircase in \cite{KessebStratmann:07}.
For a \emph{slippery devil's staircase} we have that although the
underlying measure is still singular with respect to $\lambda$, the
support of the measure is equal to an interval. As was shown in \cite{KessebStratmann:07},
the measure of maximal entropy $m_{U}$ for the Farey map $U$ has
a distribution function which is a slippery devil's staircase. In
fact, this distribution function is equal to Minkowski's Question
Mark Function. More precisely, the main results in \cite{KessebStratmann:07}
for this particular slippery devil's staircase are, and the reader
is asked to compare these with the outcome for ordinary devil's staircases
in (\ref{devil}), \[
\dim_{H}(m_{U})<\dim_{H}(\Lambda_{\sim})=\dim_{H}\left(\Lambda_{\infty}\right)<\dim_{H}\left(\Lambda_{0}\right)=1.\]

\vspace{2mm}

\noindent \textbf{III.} Let us end this introduction with a brief
discussion of two special cases of our Main Theorem, namely the one
in which $\nu_{\psi}$ is an Ahlfors regular measure and the one in
which $\nu_{\psi}$ is a self-similar measure of the type considered
by Darst.

\subsection*{\emph{The Ahlfors regular case.}}

Recall that a measure $\mu$ is called $t$-Ahlfors regular if and
only if $\mu(B(\xi,r))\asymp r^{t}$, for all $0<r<r_{0}$ and for
all $\xi$ in the support of $\mu$, for some fixed $r_{0},t>0$.
In the situation of the Main Theorem, it is well-known that if the
Gibbs measure $\nu_{\psi}$ is $t$-Ahlfors regular, then $t$ is
equal to the Hausdorff dimension $\delta$ of $\mathcal{L}$ and $\psi=\delta\phi$.
In this case, the Hausdorff dimension $s$ of $\Lambda_{\psi}^{\alpha}$
can be calculated explicitly for $\alpha>\delta$ as follows. Namely,
here (\ref{pressure}) implies $P\left(\left(s+\left(\delta-\alpha\right)\beta_{\alpha}\left(s\right)\right)\phi\right)=0$,
which immediately gives $s+\left(\delta-\alpha\right)\beta_{\alpha}\left(s\right)=\delta$,
and hence,%
\begin{figure}
\includegraphics[clip,width=.95\textwidth]{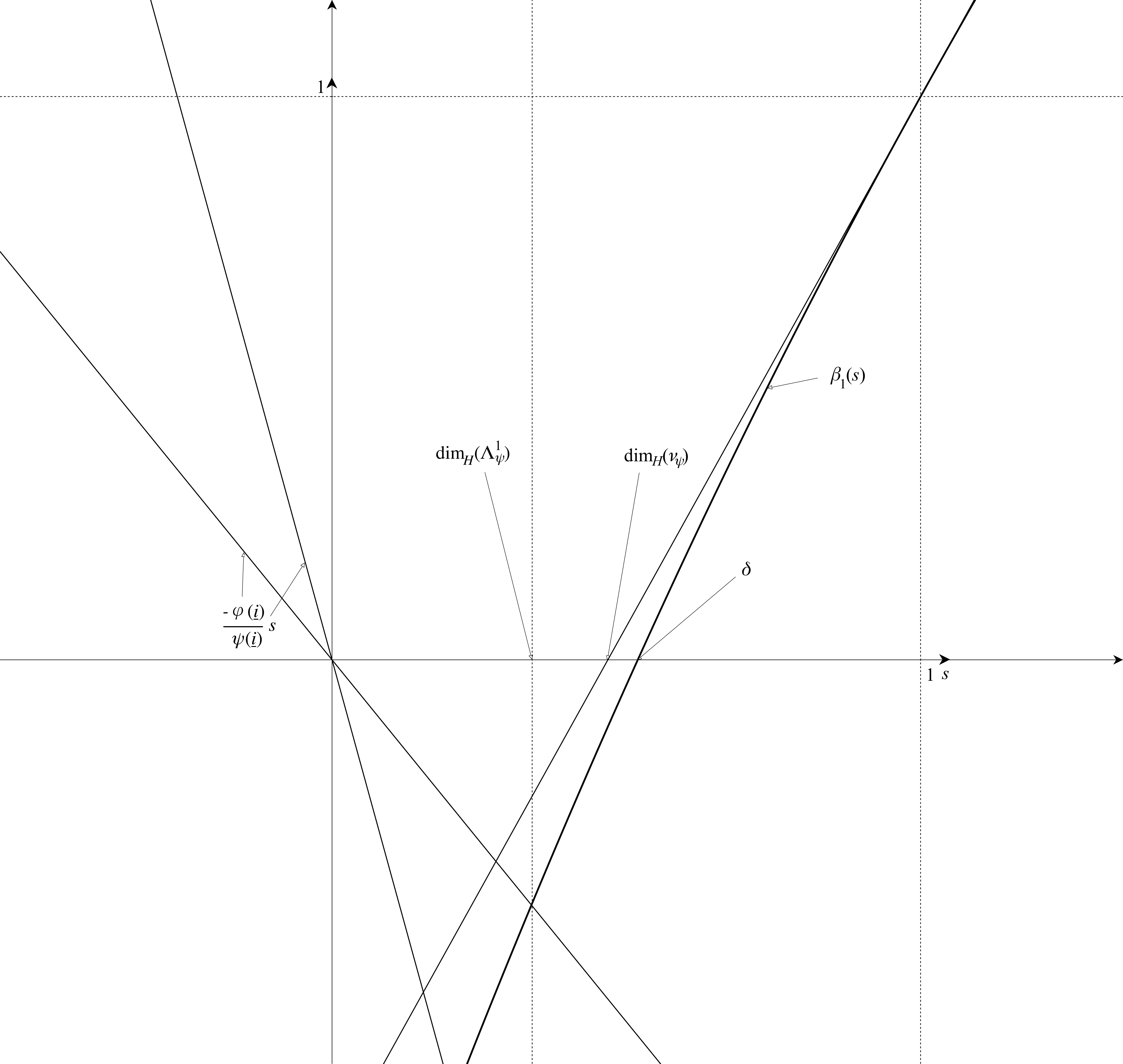}

\caption{\label{fig:Darst1}The graph of $\beta_{\alpha}$ for the Darst case
with $a_{1}=0.1$ and $a_{2}=0.5$ and $\alpha=1$. In this case $\dim_{H}\left(\nu_{\psi}\right)>\dim_{H}\left(\Lambda_{\psi}^{1}\right)$
(cf. Main Theorem and Proposition \ref{pro:tangent}).}

\end{figure}

\begin{figure}
\includegraphics[width=0.9\columnwidth]{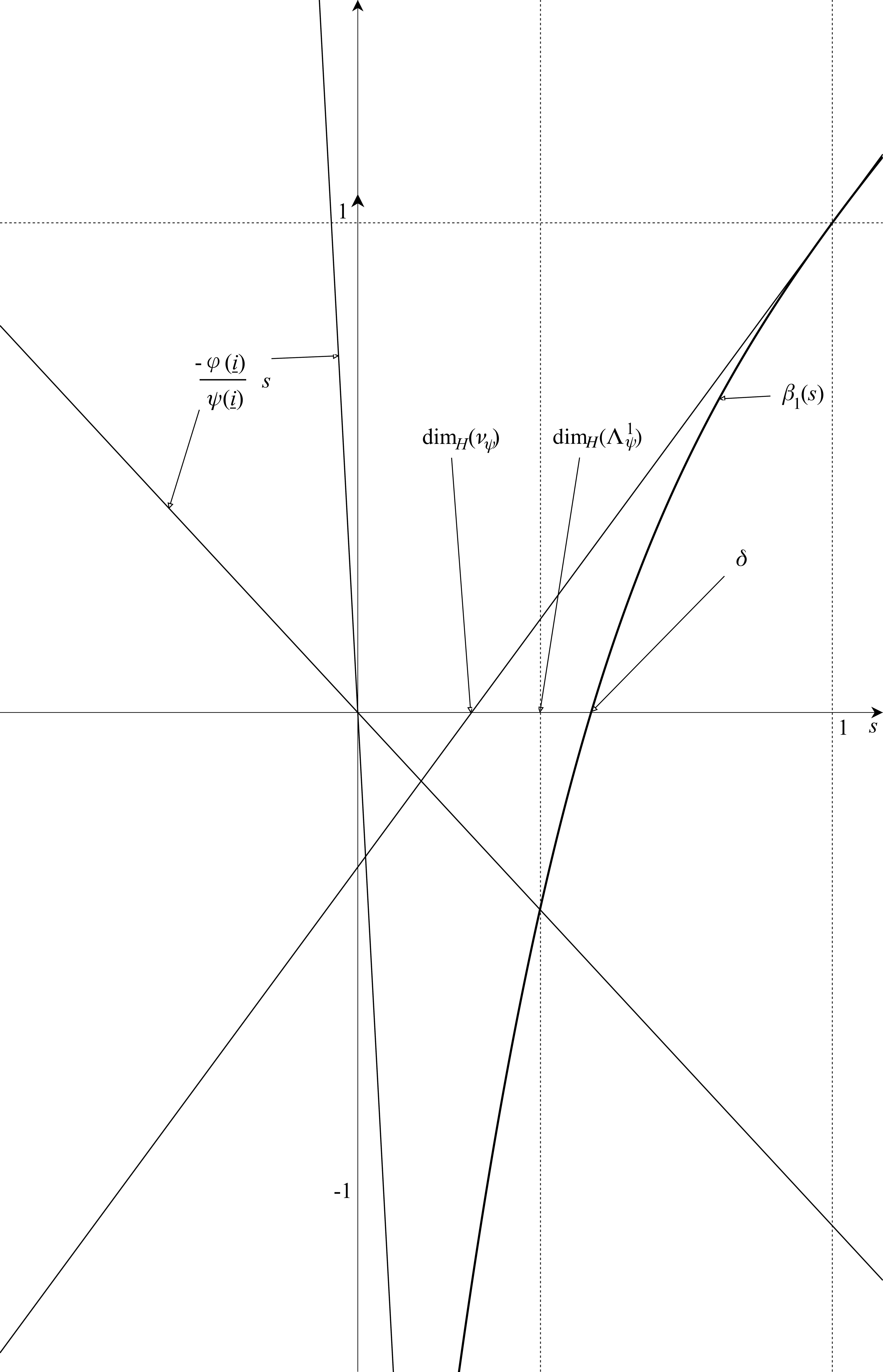}

\caption{\label{fig:Darst2} The graph of $\beta_{\alpha}$ for the Darst case
with $a_{1}=0.01$, $a_{2}=0.8$ and $\alpha=1$. In this case $\dim_{H}\left(\nu_{\psi}\right)<\dim_{H}\left(\Lambda_{\psi}^{1}\right)$
(cf. Main Theorem and Proposition \ref{pro:tangent}).}

\end{figure}

\[
\beta_{\alpha}\left(s\right)=\frac{\delta-s}{\delta-\alpha}.\]
 Inserting this into (\ref{eq:linRel}), we obtain $s+\delta(\delta-s)/(\delta-\alpha)=0$.
Solving the latter for $s$, one rediscovers Falconer's result \cite{Falconer:04}
on the Hausdorff dimension of $\Lambda_{\delta\phi}^{\alpha}$, namely
\begin{eqnarray}
\dim_{H}(\Lambda_{\delta\phi}^{\alpha})=\frac{\delta^{2}}{\alpha},\,\,\hbox{for\, all}\,\,\alpha>\delta.\label{formula}\end{eqnarray}
 Let us point out that, as also noted in \cite{Falconer:04}, the
equality in (\ref{formula}) remains to be true for $\alpha=\delta$.
However, this case requires some additional care, such as for instance
the use of ergodicity of the measure $\nu_{\delta\phi}$ or alternatively
some Khintchine-type argument, and hence let us not go into the details
here. Also, note that (\ref{formula}) in particular includes the
result of Darst \cite{Darst:93}, who only considered the case $\alpha=1$
for Cantor sets and showed that in this special linear situation one
has \[
\dim_{H}(\Lambda_{\delta\phi}^{1})=\delta^{2}.\]

\subsection*{\emph{The Darst self-similar case.}}

Here, only the case $\alpha=1$ has previously been considered in
the literature. For instance, in his studies of Cantor sets Darst
\cite{Darst:95} considered self-similar measures with probabilities
$p_{1}:=a_{1}/(a_{1}+a_{2})$ and $p_{2}:=a_{2}/(a_{1}+a_{2})$, where
$a_{1}$ and $a_{2}$ refer to the contraction rates of the two similarities
generating the underlying Cantor set. This Darst scenarion is contained
as a special case in our Main Theorem. Namely, here the $f_{a}$ have
to be linear contractions and the potential function $\psi$ has to
be equal to $\phi-P(\phi)$. We then have that the Hausdorff dimension
$s$ of $\Lambda_{\delta\phi}^{\alpha}$ and $\beta_{1}(s)$ can be
calculated explicitly. Indeed, here (\ref{pressure}) implies $P\left(s\phi-\beta_{1}(s)P\left(\phi\right)\right)=0$,
which gives \[
\beta_{1}\left(s\right)=\frac{P\left(s\phi\right)}{P\left(\phi\right)}.\]
 Inserting this into (\ref{eq:linRel}) gives that $s$ is the unique
solution of the pressure equation \[
sP(\phi)=P(s\phi)\left(\frac{P(\phi)}{\min\{\phi(\underline{i}):i\in\{0,1\}\}}-1\right).\]
 Besides Darst, this self-similar situation has previously also been
considered by Li \cite{Li:07}, Li $\&$ Xiao $\&$ Dekking \cite{LiXiaoDekking:02},
and Morrison \cite{Mo:02}.

\textbf{Acknowledgement:} We would like to thank the Mathematische
Forschungsinstitut Oberwolfach for its support and for the excelent
research enviroment. We began with the work towards this paper during
our two weeks Research-in-Pairs visit to Oberwolfach.

\section{Preliminaries\label{section2}}

\subsection{Thermodynamic formalism for iterated function systems}

Throughout, we will consider the following type of conformal iterated
function systems $\mathcal{F}$. For some compact connected set $X\subset\R$
and with $A:=\{0,1,\ldots,d\},$ let $\mathcal{F}=\{f_{a}:\, a\in A\}$
be generated by differentiable contractions $f_{a}:X\to\Int X$ such
that the following two conditions hold.

\vspace{2mm}

\noindent {\em{Strong separation condition}}. $\,$ $f_{a}(\Int(X))\subset\Int(X)$
for all $a\in A$, and $f_{a}(X)\cap f_{b}(X)=\emptyset$, for each
pair of distinct $a,b\in A$.

\vspace{2mm}

\noindent {\em{Hölder condition}}. There exists $\epsilon>0$
and an open interval $Y\supset X$ such that $f_{a}$ has a $C^{1+\epsilon}-$continuation
$\tilde{f}_{a}$ to $Y$ for which $\tilde{f}_{a}\left(Y\right)\subset Y$
and $\tilde{f}_{a}:Y\to\tilde{f}_{a}\left(Y\right)$ is a diffeomorphism,
for each $a\in A$.

\vspace{2mm}
 Note that the Hölder condition immediately implies the bounded distortion
property. That is, we in particular have that $\mathcal{F}$ has the
following property.\vspace{2mm}

\noindent {\em{Bounded distortion property}}. $\,$ For each
$\omega\in A^{n},n\in\N$ and $\xi,\eta\in X$, we have \[
|f_{\omega}'(\xi)|\asymp|f_{\omega}'(\eta)|.\]
\vspace{2mm}
Here we have used the notation $f_{\omega}:=f_{x_{1}}\circ f_{x_{2}}\circ\ldots\circ f_{x_{n}}$
for $\omega=x_{1}x_{2}\ldots x_{n}\in A^{n}$. Without loss of generality,
we will always assume that the intervals $\{f_{a}(X):a\in A\}$ are
labeled as follows. If $\xi\in f_{0}(X)$ and $\eta\in f_{a}(X)$
for some $a\in A\setminus\{0\}$, then $\xi<\eta$. Likewise, if $\xi\in f_{1}(X)$
and $\eta\in f_{a}(X)$ for some $a\in A\setminus\{1\}$, then $\xi>\eta$.
In other words, $f_{0}(X)$ ($f_{1}(X)$ resp.) is assumed to be the
utter left (right resp.) interval in the first iteration level $\{f_{a}(X):a\in A\}$
of $\mathcal{F}$.

\noindent Recall that the limit set $\mathcal{L}$ of $\mathcal{F}$
is the unique non-empty compact subset of $\R$ which satisfies $\mathcal{L}=\bigcup_{a\in A}f_{a}(\mathcal{L})$.
Equivalently, $\mathcal{L}$ is given by \[
\mathcal{L}:=\bigcap_{n\in\N}\bigcup_{\omega\in A^{n}}f_{\omega}(X).\]
 Clearly, the latter description of $\mathcal{L}$ immediately shows
that each element of $\mathcal{L}$ can be coded in a unique way by
an infinite word with letters chosen from the alphabet $A$. That
is, there is a bijective coding map $\Phi:A^{\N}\to\mathcal{L}$,
which is given by \[
\Phi:(x_{1},x_{2},\ldots)\mapsto\bigcap_{n\in\N}f_{x_{1}\ldots x_{n}}(X).\]
 For ease of exposition, we will make no explicit destinction between
$(x_{1},x_{2},\ldots)\in A^{\N}$ and $\xi:=\Phi((x_{1},x_{2},\ldots))\in\mathcal{L}$.
Also, throughout we assume that the reader is familiar with the following
basic concepts of the thermodynamic formalism (see e.g. \cite{Bowen:75},
\cite{Denker:05}, \cite{Pesin:97}, \cite{Ruelle:78}), where we
use the common notation for cylinder sets $[x_{1}\ldots x_{n}]:=\{y=(y_{1}y_{2}\ldots)\in\mathcal{L}:y_{i}=x_{i},\hbox{\, for\, all\,}\, i\in\{1,\ldots,n\}\}$,
as well as the notation for Birkhoff sums $S_{n}g:=\sum_{k=0}^{n-1}g\circ\sigma^{k}$
with $\sigma$ referring to the usual left-shift map on $A^{\N}$.

\begin{itemize}
\item The canonical geometric potential $\phi:\mathcal{L}\to\R$ associated
with $\mathcal{F}$ is given by \[
\phi(\xi):=\log f_{x_{1}}'(\xi),\,\hbox{for\, all}\,\,\xi=(x_{1}x_{2}\ldots)\in\mathcal{L}.\]
 Note that, since $\mathcal{F}$ satisfies the Hölder condition, $\phi$
is Hölder-continuous. 
\item The pressure function $P:C\left(\mathcal{L},\R\right)\to\R$ is given
for continuous potential functions $g:\mathcal{L}\to\R$ by \[
P(g):=\lim_{n\to\infty}\frac{1}{n}\log\sum_{\omega\in A^{n}}\exp(\sup_{\xi\in[\omega]}S_{n}g(\xi)).\]

\item Throughout, let $\psi\in C\left(\mathcal{L},\R\right)$ refer to a
given Hölder-continuous function for which $\psi<0$, $P(\psi)=0$,
and $\psi>\alpha\phi$. Then it is well-known that there is a Gibbs
measure $\nu_{\psi}$ associated with $\psi$ such that \[
\nu_{\psi}([\omega])\asymp e^{S_{n}\psi(\xi)},\,\,\hbox{for\, all}\,\,\xi\in[\omega],\omega\in A^{n},n\in\N.\]

\item Also, throughout we let $\chi_{\alpha}:=\psi-\alpha\phi>0$, for some
$\alpha\in\R$, and then consider the potential function \[
s\phi+t\chi_{\alpha},\,\,\hbox{for}\,\, s,t\in\R.\]
 By standard thermodynamic formalism, there then exists a strictly
increasing, concave, and real-analytic function $\beta_{\alpha}:\R\to\R$
such that \[
P(s\phi+\beta_{\alpha}(s)\chi_{\alpha})=0,\,\,\hbox{for\, all}\,\, s\in\R.\]
 The Gibbs measure associated with the potential function $s\phi+\beta_{\alpha}(s)\chi_{\alpha}$
will be denoted by $\mu_{s}$. Note that the measure $\mu_{s}$ satisfies
\[
\mu_{s}([\omega])\asymp e^{sS_{n}\phi(\xi)+\beta_{\alpha}(s)S_{n}\chi_{\alpha}(\xi)},\,\,\hbox{for\, all}\,\,\xi\in[\omega],\omega\in A^{n},n\in\N.\]

\end{itemize}
\noindent The following proposition in particular shows in which way
the values of $\delta$ and $\dim_{H}(\nu_{\psi})$ can be obtained
from the graph of the function $\beta_{\alpha}$. We refer to Fig.
\ref{fig:Darst1} and Fig. \ref{fig:Darst2} for illustrations of
this proposition for the special case $\alpha=1$.

\begin{prop}
\label{pro:tangent} We have that the unique zero of $\beta_{\alpha}$
is at $\delta$ and that $\beta_{\alpha}\left(\alpha\right)=1$. Moreover,
$\dim_{H}\left(\nu_{\psi}\right)$ is the point of intersection of
the $x$--axis with the tangent of the graph of $\beta_{\alpha}$
at the point $\left(\alpha,1\right)$. 
\end{prop}
\begin{proof}
Since $P\left(\delta\phi+0\chi_{\alpha}\right)=P\left(\delta\phi\right)=0$,
we necessarily have $\beta_{\alpha}\left(\delta\right)=0$. Also ,since
$P\left(\alpha\phi+\chi_{\alpha}\right)=P\left(\alpha\phi+\psi-\alpha\phi\right)=P\left(\psi\right)=0,$
we necessarily have that $\beta_{\alpha}\left(\alpha\right)=1$. Therefore,
by employing well-known identities from the thermodynamic formalism
(see e.g. \cite{Denker:05}), it follows \[
\beta_{\alpha}'\left(1\right)=-\frac{\int\phi\, d\nu_{\psi}}{\int\chi_{\alpha}\, d\nu_{\psi}}=\frac{1}{\alpha-\int\psi\, d\nu_{\psi}/\int\phi\, d\nu_{\psi}}=\frac{1}{\alpha-\dim_{H}\left(\nu_{\psi}\right)}.\]
 This shows that the tangent $L_{(\alpha,1)}$ of the graph of $\beta_{\alpha}$
at $\left(\alpha,1\right)\in\R^{2}$ is given by \[
L_{(\alpha,1)}(x):=\frac{1}{\alpha-\dim_{H}\left(\nu_{\psi}\right)}x-\frac{\alpha}{\alpha-\dim_{H}\left(\nu_{\psi}\right)}+1.\]
 By solving the equation $L_{(\alpha,1)}(x)=0$ for $x$, the assertion
of the proposition follows. 
\end{proof}

\subsection{The derivative of the Gibbs distribution function }

As noted in the introduction, let \[
F_{\psi}:\R\to[0,1],\, x\mapsto\nu_{\psi}((-\infty,x]),\]
 refer to the distribution function of the Gibbs measure $\nu_{\psi}$
associated with the iterated function system $\mathcal{F}$ and the
potential function $\psi$. We always assume that $\mathcal{F}$ and
$\psi$ are chosen as specified in the previous section. Also, define
\[
\mathcal{E}:=\{(x_{1}x_{2}\ldots)\in\mathcal{L}:\,\hbox{there\, exist}\,\, i\in\{0,1\},n\in\N\,\,\hbox{such\, that}\,\, x_{k}=i\,\,\hbox{forall}\, k\geq n\},\]
 and let $\mathcal{L}^{*}:=\mathcal{L}\setminus\mathcal{E}$ refer
to the limit set of $\mathcal{F}$ without the countable set of `end
points' whose code eventually has only either $0$'s or $1$'s.

\begin{lem}
\label{lem0} For the upper $\alpha$-Hölder derivative of $F_{\psi}$
we have \[
\limsup_{\eta\to\xi}\frac{|F_{\psi}(\xi)-F_{\psi}(\eta)|}{|\xi-\eta|^{\alpha}}=\infty,\,\,\hbox{for\, all}\,\,\xi\in\mathcal{L}^{*}.\]

\end{lem}
\begin{proof}
Let $\xi=(x_{1}x_{2}\ldots)\in\mathcal{L}^{*}$ be given. We then
have $\xi\in[x_{1}\ldots x_{n}]$, for each $n\in\N$. Also, since
$\xi$ is not in $\mathcal{E}$, there exists $(n_{m})_{m\in\N}$
such that $x_{n_{m}+1}\neq1$, for all $m\in\N$. In this situation
we then have $\xi\notin[x_{1}\ldots x_{n_{m}}1]$. Moreover, note
that $\nu_{\psi}([x_{1}\ldots x_{n_{m}}1])\asymp\nu_{\psi}([x_{1}\ldots x_{n_{m}}])$.
Using this and the bounded distortion property, it follows, with $\eta_{m}$
referring to the right endpoint of $[x_{1}\ldots x_{n_{m}}]$, that
is $\eta_{m}:=(x_{1}\ldots x_{n_{m}}\underline{1})\in\mathcal{E}$,
\begin{eqnarray*}
\frac{|F_{\psi}(\xi)-F_{\psi}(\eta_{m})|}{|\xi-\eta_{m}|^{\alpha}} & \geq & \frac{\nu_{\psi}([x_{1}\ldots x_{n_{m}}1])}{\diam([x_{1}\ldots x_{n_{m}}])^{\alpha}}\gg\frac{\nu_{\psi}([x_{1}\ldots x_{n_{m}}])}{\diam([x_{1}\ldots x_{n_{m}}])^{\alpha}}\\
 & \gg & \frac{\exp(S_{n_{m}}\psi(\xi))}{\exp(S_{n_{m}}\alpha\phi(\xi))}=e^{S_{n_{m}}\chi_{\alpha}(\xi)}\geq e^{n_{m}\inf_{\eta\in\mathcal{L}}\chi_{\alpha}(\eta)}.\end{eqnarray*}
 Since $\chi_{\alpha}>0,$ the result follows. 
\end{proof}
\begin{defn*}
\noindent Let us say that $\xi=(x_{1}x_{2}\ldots)\in\mathcal{L}$
has an $i$\emph{-block of length $k$ at the $n$-th level}, for
$n,k\in\N$ and $i\in\{0,1\}$, if $x_{n},x_{n+k+1}\in A\setminus\{i\}$
and $x_{n+m}=i$, for all $m\in\{1,\ldots,k\}$. 
\end{defn*}
\begin{prop}
\label{Prop1} If for $i\in\{0,1\}$ we have that $\xi=(x_{1}x_{2}\ldots)\in\mathcal{L}$
has an $i$-block of length $k$ at the $n$-th level, then there
exists $\eta\in\mathcal{L}$ such that $|\xi-\eta|\asymp\exp(S_{n}\phi(\xi))$,
and \[
\frac{|F_{\psi}(\xi)-F_{\psi}(\eta)|}{|\xi-\eta|^{\alpha}}\asymp\e^{S_{n}\chi_{\alpha}(\xi)}\cdot\e^{k\psi(\underline{i})}.\]

\end{prop}
\begin{proof}
Let $\xi=(x_{1}x_{2}\ldots)\in\mathcal{L}$ be given as stated in
the lemma. Let us only consider the case $i=1$. The case $i=0$ can
be dealt with in a similar way, and this is left to the reader. We
then have that there exists $j\in A$ such that the interval $J$
bounded by the points $\eta:=(x_{1}\ldots x_{n-1}j\underline{0})\in\mathcal{E}$
and $\hat{\eta}:=(x_{1}\ldots x_{n}\underline{1})\in\mathcal{E}$
is a `gap interval' in the construction of $\mathcal{L}$ such that
$J\cap\mathcal{L}^{*}=\emptyset$ (that is, $J$ denotes the gap interval
in the construction of $\mathcal{L}$ separating $[x_{1}\ldots x_{n}]$
and its right neighbour in the $n$-th level). Using the bounded distortion
property and the strong seperation condition, we then have \[
|\xi-\eta|\asymp\diam(J)\asymp\diam([x_{1}\ldots x_{n}])\asymp\e^{S_{n}\phi(\xi)}.\]
 Moreover, since $\nu_{\psi}(J)=0$, we have \begin{eqnarray*}
|F_{\psi}(\xi)-F_{\psi}(\eta)| & = & \nu_{\psi}((\xi,\eta])=\nu_{\psi}((\xi,\hat{\eta}])\leq\nu_{\psi}([x_{1}\ldots x_{n+k}])\\
 & \ll & \e^{S_{n+k}\psi(\xi)}\ll\e^{S_{n}\psi(\xi)}\cdot\e^{k\psi(\underline{1})}.\end{eqnarray*}
 Finally, by noting that $\xi\notin[x_{1}\ldots x_{n+k}1]$ and $[x_{1}\ldots x_{n+k}1]\subset(\xi,\hat{\eta}]$,
we similarly have \[
|F_{\psi}(\xi)-F_{\psi}(\eta)|=\nu_{\psi}((\xi,\hat{\eta}])\geq\nu_{\psi}([x_{1}\ldots x_{n+k}1])\gg\e^{S_{n+k+1}\psi(\xi)}\gg\e^{S_{n}\psi(\xi)}\cdot\e^{k\psi(\underline{1})}.\]
 Combining these observations, it follows \[
\frac{|F_{\psi}(\xi)-F_{\psi}(\eta)|}{|\xi-\eta|^{\alpha}}\asymp\frac{\exp(S_{n}\psi(\xi)+k\psi(\underline{1}))}{\exp(S_{n}\alpha\phi(\xi))}=e^{S_{n}\chi_{\alpha}(\xi)}\cdot e^{k\psi(\underline{1})}.\]

\end{proof}
\noindent The following corollary gives a generalization of a classical
result of Gilman \cite{Gilman:32}, who showed for Cantor sets that
if the derivative of the Cantor function exists in the generalized
sense at some point in the Cantor set, then it has to be equal to
infinity.

\begin{cor}
\label{cor1} If $D^{\alpha}F_{\psi}$ exists in the generalized sense
at $\xi\in\mathcal{L}^{*}$, then $(D^{\alpha}F_{\psi})(\xi)=\infty$.
On the other hand, there \emph{are} (plenty of) points in $\mathcal{L}^{*}$
at which $D^{\alpha}F_{\psi}$ does not exist in the generalized sense. 
\end{cor}
\begin{proof}
The first assertion is an immediate consequence of Lemma \ref{lem0}.
For the second assertion, choose strictly increasing sequences $(n_{m})_{m\in\N}$
and $(k_{m})_{m\in\N}$ such that $n_{m+1}>n_{m}+k_{m}$, and let
$\xi$ be an element of $\mathcal{L}$ which has an $i$-block of
length $k_{m}$ at the $n_{m}$-th level. Moreover, assume that $(n_{m})$
and $(k_{m})$ are chosen such that $\exp(S_{n_{m}}\chi_{\alpha}(\xi)+k_{m}\psi(\underline{i}))\ll1$.
Using Lemma \ref{Prop1}, it follows that there exists $(\eta_{m})_{m\in\N}$
such that $\lim_{m\to\infty}\eta_{m}=\xi$ and $|F_{\psi}(\xi)-F_{\psi}(\eta_{m})|/|\xi-\eta_{m}|^{\alpha}\ll1$,
for all $m\in\N$. From this we immediately deduce that for the lower
$\alpha$-Hölder derivative of $F_{\psi}$ we have \[
\liminf_{\eta\to\xi}\frac{|F_{\psi}(\xi)-F_{\psi}(\eta)|}{|\xi-\eta|^{\alpha}}<\infty.\]
 By combining the latter with Lemma \ref{lem0}, it follows that $(D^{\alpha}F_{\psi})(\xi)$
does not exist in the generalized sense. This finishes the proof. 
\end{proof}
\noindent For later use we also state the following immediate consequence
of Proposition \ref{Prop1}.

\begin{cor}
\label{cor2} Let $\xi=(x_{1}x_{2}\ldots)\in\mathcal{L}$ be given
such that for some $n,k\in\N$ and $i\in\{0,1\}$ we have $x_{n+m}=i$,
for all $m\in\{1,\ldots,k\}$. Then there exists $\eta\in\mathcal{L}$
such that $|\xi-\eta|\asymp\exp(S_{n}\phi(\xi))$ and \[
\frac{|F_{\psi}(\xi)-F_{\psi}(\eta)|}{|\xi-\eta|^{\alpha}}\ll\e^{S_{n}\chi_{\alpha}(\xi)}\cdot\e^{k\psi(\underline{i})}.\]

\end{cor}
\begin{prop}
\label{Prop3} We have that $D^{\alpha}F_{\psi}$ does not exist in
the generalized sense at $\xi\in\mathcal{L}^{*}$ if and only if there
exists $i\in\left\{ 0,1\right\} $ and strictly increasing sequences
$(n_{m})_{m\in\N}$ and $(k_{m})_{m\in\N}$ of positive integers such
that $\xi$ has an $i$-block of length $k_{m}$ at the $n_{m}$-th
level for each $m\in\N$, and \[
e^{S_{n_{m}}\chi_{\alpha}(\xi)+k_{m}\psi(\underline{i})}\ll1,\,\,\hbox{for\, each}\,\, m\in\N.\]

\end{prop}
\begin{proof}
The `if-part' follows immediately from combining Lemma \ref{lem0}
and Proposition \ref{Prop1}. For the `only-if-part', assume by way
of contradiction that $\xi=(x_{1}x_{2}\ldots)\in\mathcal{L}^{*}$
is given such that if $\xi$ has a strictly increasing sequences of
$i$-blocks, say of length $k_{m}$ at the $n_{m}$-th level, for
some $i\in\{0,1\}$, then \[
\liminf_{m\to\infty}e^{S_{n_{m}}\chi_{\alpha}(\xi)+k_{m}\psi(\underline{i})}=\infty.\]
 Let $(\xi_{m})_{m\in\N}$ be any sequence in $X\setminus\{\xi\}$
such that $\lim_{m\to\infty}\xi_{m}=\xi$. The aim is to show that
under these assumptions we necessarily have \[
\lim_{m\to\infty}\frac{|F_{\psi}(\xi)-F_{\psi}(\xi_{m})|}{|\xi-\xi_{m}|}=\infty.\]
 For this, first note that we can assume without loss of generality
that $\xi_{m}\in\mathcal{L}$. Indeed, if $\xi_{m}\notin\mathcal{L}$,
then move $\xi_{m}$ away from $\xi$ until one first hits $\mathcal{L}$,
say at the point $\xi_{m}'\in\mathcal{L}$ (note that by choosing
$\xi_{m}$ sufficiently close to $\xi$, we can assume without loss
of generality that such a $\xi_{m}'$ always exists, since $\xi\in\mathcal{L}^{*}$).
Clearly, we then have $|F_{\psi}(\xi)-F_{\psi}(\xi_{m})|=|F_{\psi}(\xi)-F_{\psi}(\xi_{m}')|$
as well as $|\xi-\xi_{m}|\leq|\xi-\xi_{m}'|$, and hence, \[
\frac{|F_{\psi}(\xi)-F_{\psi}(\xi_{m})|}{|\xi-\xi_{m}|^{\alpha}}\geq\frac{|F_{\psi}(\xi)-F_{\psi}(\xi_{m}')|}{|\xi-\xi_{m}'|^{\alpha}}.\]
 For ease of exposition, let us now only consider the case in which
$\xi_{m}>\xi$ for all $m\in\N$. The case $\xi_{m}<\xi$ can be dealt
with in the same way, and this is left to the reader. Now, let $(n_{m})_{m\in\N}$
be the sequence such that $\xi_{m}\in[x_{1}\ldots x_{n_{m}-1}]$ and
$\xi_{m}\notin[x_{1}\ldots x_{n_{m}-1}x_{n_{m}}]$. Then there are
two cases to consider. The first case is that $x_{n_{m}+1}\neq1$,
and the second case is that $\xi$ has a 1-block of length $k_{m}$
at the $n_{m}$-th level. For these two cases one argues as follows. 
\begin{description}
\item [{{{Case~1}}}] Here we have that $\left[x_{1}\ldots x_{n_{m}}1\right]$
separates the points $\xi$ and $\xi_{m}$, and hence, \[
\frac{|F_{\psi}(\xi)-F_{\psi}(\xi_{m})|}{|\xi-\xi_{m}|^{\alpha}}\geq\frac{\nu_{\psi}\left(\left[x_{1}\ldots x_{n_{m}}1\right]\right)}{\diam\left(\left[x_{1}\ldots x_{n_{m}-1}\right]\right)^{\alpha}}\gg e^{S_{n_{m}}\chi_{\alpha}\left(\xi\right)}.\]

\item [{{{Case~2}}}] Here we have that $\left[x_{1}\ldots x_{n_{m}}\underline{1}_{k_{m}+1}\right]$
separates the points $\xi$ and $\xi_{m}$, where $\underline{1}_{k}$
refers to the word of length $k$ containing exclusively the letter
$1$. In this situation we obtain\[
\frac{|F_{\psi}(\xi)-F_{\psi}(\xi_{m})|}{|\xi-\xi_{m}|^{\alpha}}\geq\frac{\nu_{\psi}\left(\left[x_{1}\ldots x_{n_{m}}\underline{1}_{k_{m}+1}\right]\right)}{\diam\left(\left[x_{1}\ldots x_{n_{m}-1}\right]\right)^{\alpha}}\gg e^{S_{n_{m}}\chi_{\alpha}\left(\xi\right)+k_{m}\psi\left(\underline{1}\right)}.\]

\end{description}
In both cases we have that the right hand side is unbounded. This
proves that the right $\alpha$-Hölder derivative of $F_{\psi}$ at
$\xi$ does exist in the generalized sense. By proceeding similarly
for the left $\alpha$-Hölder derivative of $F_{\psi}$ (where one
essentially has to take the `mirror image' of the above argument and
to replace $1$ by $0$), the statement of the proposition follows. 
\end{proof}

\section{Proof of the Main Theorem }

In this section we give the proof of the Main Theorem. We split up
the proof by first giving the proof for the upper bound of the Hausdorff
dimension of $\Lambda_{\psi}^{\alpha}$, and this is then followed
by the proof of the lower bound.

\noindent Throughout, let us fix for each $n\in\N$ a partition $\mathcal{C}_{n}$
of $\mathcal{L}$ by cylinder sets such that the following holds.
\[
\hbox{For\, each}\,\,[\omega]\in\mathcal{C}_{n}\,\,\hbox{and}\,\,\xi\in[\omega],\,\,\hbox{we\, have}\,\,|S_{|\omega|}\chi_{\alpha}(\xi)-n|\ll1,\]
 where $|\omega|$ refers to the word length of $\omega$. In the
following we also consider the `stopping time' $T_{t}:\mathcal{L}\to\R$,
which is defined by \[
T_{t}(\xi):=\sup\{k\in\mathbb{N}:S_{k}\chi_{\alpha}(\xi)<t\},\;\mbox{ for all }\; t>0,\xi\in\mathcal{L}.\]
 Moreover, for $i\in\{0,1\}$ and $\epsilon>0$ we define \[
\mathcal{C}_{n}^{(i)}\left(\epsilon\right):=\left\{ [\omega\underline{i}_{n_{\epsilon}}]:[\omega]\in\mathcal{C}_{n}\right\} ,\]
 where $\underline{i}_{n_{\epsilon}}$ refers to the word of length
$n_{\epsilon}:=\left\lfloor -n\left(1-\epsilon\right)/\psi(\underline{i})\right\rfloor $
containing exclusively the letter $i$.

\noindent Finally, for each $i\in\{0,1\}$ let $s_{i}$ be the unique
solution of the equation \begin{equation}
s_{i}\phi(\underline{i})/\psi(\underline{i})+\beta_{\alpha}(s_{i})=0.\label{main2}\end{equation}
 Throughout, let us always assume without loss of generality that
$\max\{s_{0},s_{1}\}=s_{0}$. For the proof of our Main Theorem it
is then sufficient to show that \begin{equation}
\dim_{H}(\Lambda_{\psi}^{\alpha})=s_{0}.\label{main1}\end{equation}
 Indeed, for this note that (\ref{main2}) has the following interpretation.
Namely, $s_{i}$ is equal to the $x$-coordinate of the point of intersection
of the graph of $\beta_{\alpha}$ with the straight line $L_{i}$
through the origin of slope $-\phi(\underline{i})/\psi(\underline{i})$.
Since $\beta_{\alpha}$ is increasing and $\beta_{\alpha}(t)<0$ for
$t<\delta$, the assumption $s_{0}=\max\{s_{0},s_{1}\}$ gives that
the slope of $L_{1}$ has to be less than or equal to the slope of
$L_{0}$. Hence, since $\phi\left(\underline{i}\right)/\psi\left(\underline{i}\right)>0$,
it follows that the minimum of $\phi(\underline{i})/\psi(\underline{i})$
is attained for $i=0$. Combining this observation with (\ref{main2})
and assuming that (\ref{main1}) holds, the implicit characterization
of $\dim_{H}(\Lambda_{\psi}^{\alpha})$ in (\ref{eq:linRel}) follows.

\noindent Therefore, we are now left with proving the statement in
(\ref{main1}), which will be done in the following two remaining
sections.

\subsection{The upper bound}

In this section we give the proof for the upper bound of the Hausdorff
dimension of $\Lambda_{\psi}^{\alpha}$ as stated in the Main Theorem.
In a nutshell, the idea is to show that there is a suitable covering
of $\Lambda_{\psi}^{\alpha}$ which allows to apply the Borel-Cantelli
Lemma in order to derive the desired upper bound.

Recall that $s_{0}=\max\{s_{0},s_{1}\}$, and let $\epsilon>0$ be
given. For $\kappa>0$ such that $(1-\epsilon)(s_{0}+\kappa)=s_{0}+\tau$,
for some $\tau>0$ (note that by choosing $\epsilon$ sufficiently
small, $\kappa$ can be made arbitrary small), we then have\\
 \\
 ${\displaystyle \sum_{n\in\N}\sum_{C\in\mathcal{C}_{n}^{(0)}(\epsilon)}\left(\diam(C)\right)^{s_{0}+\kappa}\asymp\sum_{n\in\N}\sum_{C\in\mathcal{C}_{n}^{(0)}(\epsilon)}\e^{\sup_{\xi\in C}(s_{0}+\kappa)S_{T_{n}(\xi)+n_{\epsilon}}\phi(\xi)}}$
\begin{eqnarray*}
 & \asymp & \sum_{n\in\N}\sum_{C\in\mathcal{C}_{n}^{(0)}(\epsilon)}\e^{\sup_{\xi\in C}(s_{0}+\kappa)S_{T_{n}(\xi)+n_{\epsilon}}\phi(\xi)}\\
 & \ll & \sum_{n\in\N}\e^{-n(1-\epsilon)(s_{0}+\kappa)\phi(\underline{0})/\psi(\underline{0})}\sum_{C\in\mathcal{C}_{n}^{(0)}}\e^{s_{0}\sup_{\xi\in C}S_{T_{n}(\xi)}\phi(\xi)}\\
 & \asymp & \sum_{n\in\N}\e^{-n(1-\epsilon)(s_{0}+\kappa)\phi(\underline{0})/\psi(\underline{0})-n\beta_{\alpha}(s_{0})}\sum_{C\in\mathcal{C}_{n}}\e^{\sup_{\xi\in C}S_{T_{n}(\xi)}\left(s_{0}\phi(\xi)+\beta_{\alpha}(s_{0})\chi_{\alpha}(\xi)\right)}\\
 & \asymp & \sum_{n\in\N}\left(\e^{-\tau\phi(\underline{0})/\psi(\underline{0})}\right)^{n}<\infty.\end{eqnarray*}
 Here, we have used the Gibbs property $\sum_{C\in\mathcal{C}_{n}}\exp(\sup_{\xi\in C}S_{T_{n}(\xi)}\left(s_{0}\phi(\xi)+\beta_{\alpha}(s_{0})\chi_{\alpha}(\xi)\right))\asymp1$
of the Gibbs measure $\mu_{s_{0}}$. Similar, one finds (using the
fact that $(1-\epsilon)(s_{0}+\kappa)\geq s_{1}+\tau$), \\
 \\
 ${\displaystyle \sum_{n\in\N}\sum_{C\in\mathcal{C}_{n}^{(1)}(\epsilon)}\left(\diam(C)\right)^{s_{0}+\kappa}\ll\sum_{n\in\N}\e^{-n(1-\epsilon)(s_{0}+\kappa)\phi(\underline{1})/\psi(\underline{1})}\sum_{C\in\mathcal{C}_{n}}\e^{s_{0}\sup_{\xi\in C}S_{T_{n}(\xi)}\phi(\xi)}}$
\begin{eqnarray*}
 & \ll & \sum_{n\in\N}\e^{-n(1-\epsilon)(s_{0}+\kappa)\phi(\underline{1})/\psi(\underline{1})-n\beta_{\alpha}(s_{1})}\sum_{C\in\mathcal{C}_{n}}\e^{\sup_{\xi\in C}S_{T_{n}(\xi)}\left(s_{1}\phi(\xi)+\beta_{\alpha}(s_{1})\chi_{\alpha}(\xi)\right)}\\
 & \ll & \sum_{n\in\N}\e^{-n(s_{1}+\tau)\phi(\underline{1})/\psi(\underline{1})-n\beta_{\alpha}(s_{1})}<\infty.\end{eqnarray*}
 Here, we have used the Gibbs property $\sum_{C\in\mathcal{C}_{n}}\exp(\sup_{\xi\in C}S_{T_{n}(\xi)}\left(s_{1}\phi(\xi)+\beta_{\alpha}(s_{1})\chi_{\alpha}(\xi)\right))\asymp1$
of the Gibbs measure $\mu_{s_{1}}$. Therefore, we now have \[
\dim_{H}(\{\xi\in X:\xi\in\bigcup_{i=0,1}\mathcal{C}_{n}^{(i)}(\epsilon)\,\,\hbox{for\, infinitely\, many}\,\, n\in\N\})\leq s_{0}+\kappa.\]
 Hence, it remains to show that for all $\epsilon>0$, \[
\Lambda_{\psi}^{\alpha}\subset\bigcup_{i=0,1}\bigcup_{n\in\N}\mathcal{C}_{n}^{(i)}\left(\epsilon\right).\]
 For this, let $\xi\in\Lambda_{\psi}^{\alpha}$ be given. By Proposition
\ref{Prop3}, there exists $i\in\left\{ 0,1\right\} $ and strictly
increasing sequences $(n_{m})_{m\in\N}$ and $(k_{m})_{m\in\N}$ of
positive integers, such that $\xi$ has an $i$-block of length $k_{m}$
at the $n_{m}$-th level, for each $m\in\N$, and \[
e^{S_{n_{m}}\chi_{\alpha}(\xi)+k_{m}\psi(\underline{i})}\ll1,\,\,\hbox{for\, each}\,\, m\in\N.\]
 By setting $\ell(n_{m}):=\lfloor S_{n_{m}}\chi_{\alpha}(\xi)\rfloor$,
it follows $\exp(k_{m})\gg\exp(-\ell(n_{m})/\psi(\underline{i}))$.
Hence, for each $\epsilon>0$ and for each $m$ sufficiently large,
we have $k_{m}\geq-\ell(n_{m})(1-\epsilon)/\psi(\underline{i}))$.
It follows that $\xi\in\mathcal{C}_{n_{m}}^{(i)}(\epsilon)$, which
finishes the proof of the upper bound.

\hspace{2mm}

\subsection{The lower bound}

In this section we give the proof for the lower bound of the Hausdorff
dimension of $\Lambda_{\psi}^{\alpha}$ as stated in the Main Theorem.
For this, we use the usual strategy and define a probability measure
$m$ supported on a certain Cantor-like set $\mathcal{M}$ contained
in $\Lambda_{\psi}$. We then show that the Mass Distribution Principle
is applicable to $(\mathcal{M},m)$, and this will then eventually
lead to the desired estimate from below.

\noindent For the construction of $\mathcal{M}$, let $\left(n_{k}\right)_{k\in\N}$
denote a rapidly increasing sequence (to be specified later) of positive
integers. With this sequence at hand, consider the sequences $\left(m_{k}\right)_{k\in\N}$
and $\left(N_{k}\right)_{k\in\N}$ which are given inductively as
follows. \\
 \\
 ${\displaystyle \,\,\,\,\, N_{1}:=n_{1}\,\,\hbox{and}\,\, N_{k}:=\left\lfloor \sum_{j=1}^{k}n_{j}+\chi_{\alpha}(\underline{0})\sum_{j=1}^{k-1}m_{j}\right\rfloor \,\,\,\hbox{for\, all}\,\,\, k\geq2,\,\,\,}$
\[
\hbox{where}\,\,\, m_{j}:=\lfloor-N_{j}/\psi(\underline{0})\rfloor\,\,\,\hbox{for\, all}\,\,\, j\in\N.\]
 Now, consider the Cantor-like set $\mathcal{M}$ given by \[
\mathcal{M}:=\{(x_{1}x_{2}\ldots)\in\mathcal{L}:(x_{1}x_{2}\ldots)=\omega_{1}\underline{0}_{m_{1}}\omega_{2}\underline{0}_{m_{2}}\ldots,\,\,\hbox{with}\,\,\omega_{j}\in\mathcal{C}_{n_{j}}\,\,\hbox{for\, all}\,\, j\in\N\},\]
 where $\underline{0}_{m_{j}}$ refers to the word consisting of $m_{j}$
times the letter $0$. In order to see that $\mathcal{M}\subset\Lambda_{\psi}^{\alpha}$,
let $\xi=(\omega_{1}\underline{0}_{m_{1}}\omega_{2}\underline{0}_{m_{2}}\ldots)\in\mathcal{M}$
be given. With $\ell_{k}$ referring to the word length of $\omega_{1}\underline{0}_{m_{1}}\ldots\omega_{k}$,
we then have by construction $\exp(S_{\ell_{k}}\chi_{\alpha}(\xi))\asymp\exp(N_{k})$.
Therefore, we have that $\xi$ has a $0$-block of length $m_{k}$
at the $\ell_{k}$-th level, for each $k\in\N$, and \[
\e^{S_{\ell_{k}}\chi_{\alpha}(\xi)+m_{k}\psi(\underline{0})}\asymp\e^{N_{k}+\lfloor-N_{k}/\psi(\underline{0})\rfloor\psi(\underline{0})}\ll1.\]
 An application of Lemma \ref{Prop3} then gives $\xi\in\Lambda_{\psi}^{\alpha}$,
showing that $\mathcal{M}\subset\Lambda_{\psi}^{\alpha}$.

\noindent The next step is to define a measure $m$ on $X$ as follows.

\begin{description}
\item [{{{(C1)}}}] For cylinder sets of the form $[\omega_{1}\underline{0}_{m_{1}}\ldots\omega_{u-1}\underline{0}_{m_{u-1}}\omega_{u}\underline{0}_{k}]$
with $k<m_{u}$ and $\omega_{j}\in\mathcal{C}_{n_{j}}$ for each $j\in\{1,\ldots,u\}$,
put \[
m([\omega_{1}\underline{0}_{m_{1}}\ldots\omega_{u-1}\underline{0}_{m_{u-1}}\omega_{u}\underline{0}_{k}]):=\prod_{j=1}^{u}\mu_{s_{0}}([\omega_{j}]).\]

\item [{{{(C2)}}}] For cylinder sets of the form $[\omega_{1}\underline{0}_{m_{1}}\ldots\omega_{u}\underline{0}_{m_{u}}x_{1}\ldots x_{l}]$
with $\left[x_{1}\ldots x_{l}\right]$ containing some cylinder set
from $\mathcal{C}_{m_{u+1}}$, as well as $\omega_{j}\in\mathcal{C}_{n_{j}}$
for all $j\in\{1,\ldots,u\}$, put \[
m([\omega_{1}\underline{0}_{m_{1}}\ldots\omega_{u}\underline{0}_{m_{u}}x_{1}\ldots x_{l}]):=\mu_{s_{0}}([x_{1}\ldots x_{l})\cdot\prod_{j=1}^{u}\mu_{s_{0}}([\omega_{j}]).\]

\end{description}
By Kolmogorov's Consistency Theorem this defines a measure which strictly
speaking is first only defined on $A^{\N}$. However, we then identify
this measure with the push-down to $\mathcal{L}$ via the coding map
$\Phi$. One immediately verifies that for the so obtained measure
$m$ on $\mathcal{L}$ we have $m\left(\mathcal{M}\right)=1$. In
order to complete the proof for the lower bound, it is now sufficient
to show that $m$ satisfies the `Frostman condition' for cylinder
sets of the type in (C1) as well as for cylinder sets of the type
in (C2). For this, first note that for $\omega\in\mathcal{C}_{n_{1}},k\leq m_{1}$,
and with $\xi$ referring to some arbitrary element in $[\omega\underline{0}_{m_{1}}]$,
we have the following estimate. \begin{eqnarray*}
m([\omega\underline{0}_{k}]) & = & \mu_{s_{0}}([\omega])\asymp\e^{s_{0}S_{T_{n_{1}}(\xi)}\phi(\xi)+n_{1}\beta_{\alpha}(s_{0})}\asymp\left(\e^{S_{T_{n_{1}}(\xi)}\phi(\xi)-n_{1}\phi(\underline{0})/\psi(\underline{0})}\right)^{s_{0}}\\
 & \asymp & \left(\e^{S_{T_{n_{1}}(\xi)}\phi(\xi)-n_{1}\phi(\underline{0})/\psi(\underline{0})}\right)^{s_{0}}\asymp\left(\e^{S_{T_{n_{1}}(\xi)+\lfloor-n_{1}/\psi(\underline{0})\rfloor}\phi(\xi)}\right)^{s_{0}}\\
 & \asymp & (\diam([\omega\underline{0}_{m_{1}}]))^{s_{0}}\leq(\diam([\omega\underline{0}_{k}]))^{s_{0}},\end{eqnarray*}
 where we have used the fact $\beta_{\alpha}(s_{0})=-s_{0}\phi(\underline{0})/\psi(\underline{0})$.
Using the latter estimate, we can now check the Frostman condition
for each of the two types of cylinder sets (C1) and (C2) separately
as follows.

\vspace{3mm}

\noindent {\em ad (C1)}: \, For cylinder sets as in (C1), we have
for each $\epsilon>0$ and with $c>0$ referring to some constant
(which takes care of the comparability constant in the above estimate
for $m([\omega\underline{0}_{k}])$ and of the fact that $\diam([ab])\ll\diam([a])\cdot\diam([b])$,
for all $a,b\in A$), \\
 \\
 ${\displaystyle m([\omega_{1}\underline{0}_{m_{1}}\ldots\omega_{u-1}\underline{0}_{m_{u-1}}\omega_{u}\underline{0}_{k}])=\prod_{j=1}^{u}\mu_{s_{0}}([\omega_{j}])}$
\begin{eqnarray*}
 & \leq & c^{u}\left(\diam([\omega_{1}\underline{0}_{\lfloor-n_{1}/\psi(\underline{0})\rfloor}])\ldots\diam([\omega_{u}\underline{0}_{\lfloor-n_{u}/\psi(\underline{0})\rfloor}])\right)^{s_{0}}\\
 & \ll & \frac{c^{u}\cdot\left(\diam([\underline{0}_{m_{u}}])\right)^{\epsilon}}{\left(\diam([\underline{0}_{m_{1}+\lfloor-m_{1}/\psi(\underline{0})\rfloor}])\ldots\diam([\underline{0}_{m_{u-1}+\lfloor-m_{u-1}/\psi(\underline{0})\rfloor}])\right)^{s_{0}}}\\
 &  & \hspace{42mm}\cdot\left(\diam([\omega_{1}\underline{0}_{m_{1}}\ldots\omega_{u}\underline{0}_{m_{u}}])\right)^{s_{0}-\epsilon},\end{eqnarray*}
 where we have used the fact

\noindent \begin{eqnarray*}
m_{j}-\lfloor-n_{j}/\psi(\underline{0})\rfloor & = & \left\lfloor -\left(\sum_{r=1}^{j}n_{r}+\chi_{\alpha}(\underline{0})\sum_{r=1}^{j-1}m_{r}\right)/\psi(\underline{0})\right\rfloor -\lfloor-n_{j}/\psi(\underline{0})\rfloor\\
 & = & m_{j-1}+\lfloor-m_{j-1}\chi_{\alpha}(\underline{0})/\psi(\underline{0})\rfloor\pm1.\end{eqnarray*}
 Now the annouced growth condition for the sequence $\left(n_{k}\right)_{n\in\N}$
comes into play. Namely, by choosing this sequence to increase sufficiently
fast, one immediately verifies that the first factor in the above
estimate is uniformly bounded from above. Indeed, we have \\
 \\
 ${\displaystyle \frac{c^{u}\cdot\left(\diam([\underline{0}_{m_{u}}])\right)^{\epsilon}}{\left(\diam([\underline{0}_{m_{1}+\lfloor-m_{1}/\psi(\underline{0})\rfloor}])\ldots\diam([\underline{0}_{m_{u-1}+\lfloor-m_{u-1}/\psi(\underline{0})\rfloor}])\right)^{s_{0}}}}$
\begin{eqnarray*}
 & \ll & c^{u}\e^{\epsilon\phi(\underline{0})m_{u}}\e^{-s_{0}\phi(\underline{0})\sum_{r=1}^{u-1}m_{r}(1-1/\psi(\underline{0}))}\\
 & = & \exp\left(\phi(\underline{0})m_{u}\left(\epsilon-\Big(\left(1-1/\psi(\underline{0})\right)s_{0}/m_{u}\sum_{r=1}^{u-1}m_{r}-u\log c/(\phi(\underline{0})m_{u})\Big)\right)\right).\end{eqnarray*}
 Hence, by choosing $\left(n_{k}\right)_{n\in\N}$ appropriately,
we can make sure that \[
\left(1-1/\psi(\underline{0})\right)\frac{s_{0}}{m_{u}}\sum_{r=1}^{u-1}m_{r}-\frac{u\log c}{\phi(\underline{0})m_{u}}<\epsilon,\]
 and thus, \[
\frac{c^{u}\cdot\left(\diam([\underline{0}_{m_{u}}])\right)^{\epsilon}}{\left(\diam([\underline{0}_{m_{1}+\lfloor-m_{1}/\psi(\underline{0})\rfloor}])\ldots\diam([\underline{0}_{m_{u-1}+\lfloor-m_{u-1}/\psi(\underline{0})\rfloor}])\right)^{s_{0}}}\ll1.\]
 Therefore, we have now shown that for each $\epsilon>0$ and for
all $u\in\N$ sufficiently large, \begin{eqnarray*}
m([\omega_{1}\underline{0}_{m_{1}}\ldots\omega_{u-1}\underline{0}_{m_{u-1}}\omega_{u}\underline{0}_{k}]) & \ll & \left(\diam([\omega_{1}\underline{0}_{m_{1}}\ldots\omega_{u}\underline{0}_{m_{u}}])\right)^{s_{0}-\epsilon}\\
 & \leq & \left(\diam([\omega_{1}\underline{0}_{m_{1}}\ldots\omega_{u}\underline{0}_{k}])\right)^{s_{0}-\epsilon}.\end{eqnarray*}
 This finishes the proof of the `Frostman condition' for cylinder
sets of type (C1).

\vspace{3mm}

\noindent {\em ad (C2)}: \, For cylinder sets as in (C2), first
note that since $\beta_{\alpha}(s_{0})<0$ and $\chi_{\alpha}>0$,
we have for all $\xi\in[x_{1}\ldots x_{l}]$, \[
\mu_{s_{0}}([x_{1}\ldots x_{l}])\asymp\e^{s_{0}S_{l}\phi(\xi)+\beta_{\alpha}(s_{0})S_{l}\chi_{\alpha}(\xi)}\leq\e^{s_{0}S_{l}\phi(\xi)}\ll\left(\diam([x_{1}\ldots x_{l}])\right)^{s_{0}}.\]
 Combining this with the estimate in {\em `ad $(C1)'$} we obtain
for each $\epsilon>0$ and for all $u\in\N$ sufficiently large, \begin{eqnarray*}
m([\omega_{1}\underline{0}_{m_{1}}\ldots\omega_{u}\underline{0}_{m_{u}}x_{1}\ldots x_{l}]) & = & \mu_{s_{0}}([x_{1}\ldots x_{l}])\cdot\prod_{j=1}^{u}\mu_{s_{0}}([\omega_{j}])\\
 & \ll & \left(\diam([x_{1}\ldots x_{l}])\right)^{s_{0}}\cdot\left(\diam([\omega_{1}\underline{0}_{m_{1}}\ldots\omega_{u}\underline{0}_{m_{u}}])\right)^{s_{0}-\epsilon}\\
 & \leq & \left(\diam([\omega_{1}\underline{0}_{m_{1}}\ldots\omega_{u}\underline{0}_{m_{u}}x_{1}\ldots x_{l}])\right)^{s_{0}-\epsilon}.\end{eqnarray*}

\vspace{3mm}

\noindent By combining the results of {\em `ad (C1)'} and {\em
`ad (C2)'} and using the fact that $\mathcal{M}$ is a subset of
$\Lambda_{\psi}^{\alpha}$, we have now shown that for each $\epsilon>0$,
\[
\dim_{H}(\Lambda_{\psi}^{\alpha})\geq\dim_{H}(\mathcal{M})\geq s_{0}-\epsilon.\]
 This finishes the proof of the lower bound, and hence also the proof
of the Main Theorem.

\end{document}